\documentclass{amsart}
\usepackage[utf8]{inputenc}
\usepackage{amsmath}
\usepackage{amsthm}
\usepackage{amsfonts}
\usepackage{amssymb}
\usepackage[all]{xy}
\usepackage{indentfirst}
\usepackage[pagebackref=true]{hyperref}

\newtheorem{thm}{Theorem}[section]
\newtheorem{prop}[thm]{Proposition}

\newtheorem{theorem}[thm]{Theorem}

\newtheorem{proposition}[thm]{Proposition}

\newtheorem{lemma}[thm]{Lemma}
\newtheorem*{theorem*}{Theorem}

\theoremstyle{definition}

\newtheorem{example}[thm]{Example}

\newtheorem*{ack}{Acknowledgement}

\newcommand{\N}{\mathbb{N}} %% Naturals
\newcommand{\Z}{\mathbb{Z}} %% Integers
\newcommand{\C}{\mathbb{C}} %% Complex
\newcommand{\Q}{\mathbb{Q}} %% Rationals
\usepackage{verbatim} %% Enables 'comment'
\usepackage{color}
\newcommand{\G}{\Gamma}
\newcommand{\g}{\gamma}
\DeclareMathOperator{\Ima}{Im}

\title{Homology of odometers}
\author{Eduardo Scarparo}

\address{Departamento de Matemática, Universidade Federal de Santa Catarina, 88040-970 Florianópolis-SC, Brazil}
\email{duduscarparo@gmail.com}

\thanks{The author was supported by CNPq, Brazil, 167983/2017-2.}

\begin{document}

\begin{abstract}
We compute the homology groups of transformation groupoids associated with odometers and show that certain $(\Z\rtimes\Z_2)$-odometers give rise to counterexamples to the HK conjecture, which relates the homology of an essentially principal, minimal, ample groupoid $G$ with the K-theory of $C^*_r(G)$. We also show that transformation grupoids of odometers satisfy the AH conjecture.

\end{abstract}

\maketitle

\section{Introduction}

Given a decreasing sequence $(\G_i)_{i\in\N}$ of finite index subgroups of a group $\G$, there is an action of $\G$ on $\varprojlim\G/\G_i$ given by left-multiplication. This action is called an odometer, and it has been extensively studied.

For example, in \cite{MR2550195}, Orfanos obtained several properties of crossed products associated to odometers, in the case that the acting group is amenable. In \cite{MR2783933}, Ioana studied these actions in the measurable setting, showing orbit equivalence superrigidity when the acting group has property (T), and, in \cite{MR3556453}, Cortez and Medynets studied continuous orbit equivalence and topological full groups of odometers.

In \cite{MR3552533}, Matui formulated two conjectures about homology groups and K-theory of second countable, étale, essentially principal, minimal groupoids with unit space homeomorphic to the Cantor set. The first one (HK conjecture) predicts that, given such a groupoid $G$, the following holds:

\begin{equation*}
K_*(C^*_r(G))\simeq\bigoplus_{k\geq 0} H_{2k+*}(G) \text{ for $*=0,1$.}
\end{equation*}
The HK conjecture has been verified for several classes of groupoids (see \cite{ortega2018homology} and \cite{farsi2018ample} for recent developments). 

The second conjecture (AH conjecture) relates the abelianization of the topological full group of $G$ with the two first homology groups of $G$. It has been verified for principal, almost finite groupoids, and groupoids arising from products of one-sided shifts of finite type (\cite{MR3552533}). 

In this note, we show that certain $(\Z\rtimes\Z_2)$-odometers are counterexamples to the HK conjecture. The particular class of  $(\Z\rtimes\Z_2)$-odometers that we consider has already appeared elsewhere (e.g., in \cite[Section 2.4]{MR3779956}, \cite[Example 7.5]{MR3506971} and \cite[Example 4.3]{scarparo2017c}). Moreover, in \cite[10.11.5(c)]{MR859867}, using a different picture, Blackadar computed the K-theory of the crossed products associated to these odometers, and, in \cite{MR962104}, Kumjian showed that these crossed products are AF (approximately finite-dimensional). 

We also show that the AH conjecture holds for transformation groupoids associated to odometers.

The paper is organized as follows. In Section \ref{pre}, we collect basic facts about odometers, and characterize when an odometer is topologically free. We also describe the K-theory and homology of odometers, and we recall the the definitions of ample groupoids and their homology groups.

In Section \ref{count}, we compute the K-theory and homology of certain $(\Z\rtimes\Z_2)$-odometers and show that they are counterexamples to the HK conjecture. In Section \ref{ah}, we show that transformation groupoids of odometers satisfy the AH conjecture.

\begin{ack}

The author thanks Kevin Aguyar Brix for inspiring conversations about homology of groupoids.

\end{ack}

\section{Preliminaries}\label{pre}

\subsection{Odometers}

Let $\G$ be a group and $(\Gamma_i)_{i\in\N}$ a sequence of finite index subgroups of $\Gamma$ such that, for every $i\in \N$, $\Gamma_i \gneq \Gamma_{i+1}$. 

For each $i\in\mathbb{N}$, let $p_i\colon \G /\Gamma_{i+1}\to\Gamma/\Gamma_{i}$ be the surjection given by 
\begin{equation}
p_i(\gamma\Gamma_{i+1}):=\gamma\Gamma_i,\text{ for } \gamma\in\Gamma.\label{proje}
\end{equation} 
Let $X:=\varprojlim(\G/\G_i,p_i)=\{(x_i)\in\prod\Gamma/\Gamma_i:p_i(x_{i+1})=x_i,\forall i\in\mathbb{N}\}$. Then $X$ is homeomorphic to the Cantor set and $\Gamma$ acts in a minimal way on $X$ by $\gamma(x_i):=(\gamma x_i)$, for $\g\in\G$ and $(x_i)\in X$. This action is called an \emph{odometer} (terminology from \cite{MR3556453}). See \cite[Fait 2.1.4]{MR3289281} and \cite[Proposition A.1]{MR3506971} for more abstract characterizations of odometers.

Note that, if $(x_i),(y_i)\in X$, and there exists $i_0$ such that $x_{i_0}=y_{i_0}$, then, for $1\leq i\leq i_0$, it holds that $x_i=y_i$.

Given $j\geq 1$ and $g\G_j\in \G/\G_j$, let $U(j,g\G_j):=\{(x_i)\in X:x_j=g\G_j\}$. Then $\{U(j,g\G_j):j\in\mathbb{N},g\G_j\in\G/\G_j\}$ is a basis for $X$ consisting of compact-open sets.

Recall that an action of a group $\G$ on a locally compact Hausdorff space $Y$ is said to be \emph{topologically free} if, for each $\g\in\G\setminus\{e\}$, the set of points of $Y$ fixed by $\g$ has empty interior.

\begin{prop} \label{topfree}
An odometer $\G\curvearrowright X:=\varprojlim\G/\G_i$ is topologically free if and only if, for every $\gamma\in\cap\Gamma_i\setminus\{e\}$ and $j\geq 1$, there exists $b\in \Gamma_j$ such that $b^{-1}\gamma b\notin\cap\Gamma_i$.

\end{prop}
\begin{proof}
Suppose that the $\G$-action is not topologically free. This implies that there exists $\gamma'\in\Gamma\setminus\{e\}$ and $V\subset X$ non-empty and open such that each $x\in V$ is fixed by $\gamma'$. Since the family $\{U(j,g\G_j):j\in\mathbb{N},g\G_j\in\G/\G_j\}$ is an open basis for $X$, there exists $j\in\mathbb{N}$ and $g\G_j\in\Gamma/\G_j$ such that $U(j,g\G_j)$ is fixed pointwise by $\gamma'$. Since $U(j,g\G_j)=gU(j,\G_j)$, it follows that $\g:=g^{-1}\g' g$ fixes pointwise $U(j,\G_j)$.

Given $b\in\Gamma_j$, we have $(b\G_i)_{i\in\N}\in U(j,\G_j)$, hence $(\gamma b\Gamma_i)=(b\Gamma_i)$. Therefore, $b^{-1}\g b\in\cap\G_i$. This concludes the backwards implication.

For the converse, assume that there exists $\gamma\in\cap\Gamma_i\setminus\{e\}$ and $j\geq 1$ such that, for every $b\in \Gamma_j$, $b^{-1}\gamma b\in\cap\Gamma_i$. 

Given $(g_i\G_i)\in U(j,\G_j)$, by the definition of $X$, we have that, for $i\geq j$, $g_i\in\G_j$. Therefore, for $i\geq j$, it holds that $\g g_i\G_i=g_i\G_i$, and this implies that $\gamma g_i\G_i =g_i\G_i$ for every $i$.

 Hence, $\g(g_i\G_i)=(g_i\G_i)$ for any $(g_i\G_i)\in U(j,\G_j)$, and the action is not topologically free.
\end{proof}

\begin{example}\label{ce}

Recall that the \emph{infinite dihedral group} is the semidirect product $\Z\rtimes\Z_2$ associated to the action of $\Z_2$ on $\Z$ by multiplication by $-1$. 

Let $(n_i)$ be a strictly increasing sequence of natural numbers such that $n_i|n_{i+1}$, for every $i\in\mathbb{N}$. Define $\Gamma:=\mathbb{Z}\rtimes\mathbb{Z}_2$ and, for $i\geq 1$, $\Gamma_i:=n_i\mathbb{Z}\rtimes\mathbb{Z}_2$. Then $\cap\G_i=\{(0,0),(0,1)\}$. Moreover, for $j\geq 1$, we have that $(n_j,0)(0,1)(-n_j,0)=(2n_j,1)\notin\cap\G_i$.  Therefore, $\G\curvearrowright\varprojlim\G/\G_i$ is topologically free.

\end{example}

Given $\G\curvearrowright\varprojlim\G/\G_i$ an odometer such that $\G_i\unlhd\G$ for every $i$, the $\G$-action is free if and only if $\cap\G_i=\{e\}$ (see \cite[Section 2.1]{MR2427048}). Furthermore, if the $\G$-action is topologically free, then Proposition \ref{topfree} implies that $\cap\G_i=\{e\}$.

Recall that an action of a group $\G$ on a set $X$ is said to be $\emph{faithful}$ if, for every $\gamma\in\G\setminus\{e\}$, there exists $x\in X$ such that $\gamma x\neq x$. The group $\G$ is said to be \emph{residually finite} if, for every $\gamma\in\G\setminus\{e\}$, there exists a finite group $F$ and a homomorphism $\varphi\colon \G\to F$ such that $\varphi(\gamma)\neq e$.

 In case there exists a faithful $\G$-odometer $\G\curvearrowright\varprojlim\G/\G_i$, then, given $\g\in\G\setminus\{e\}$, there is $j\geq 1$ such that $\g$ acts non-trivially on $\G/\G_j$. In particular, $\G$ is residually finite (this remark is from \cite{MR3672904}). Conversely, if $\G$ is countably infinite and residually finite, there exists a strictly decreasing sequence $(\G_i)$ of finite index normal subgroups of $\G$ such that $\cap\G_i=\{e\}$, and the odometer $\G\curvearrowright\varprojlim\G/\G_i$ is free.

\subsection{K-theory of odometers} Given a group $\G$ acting on a compact Hausdorff space $X$, we will denote the canonical copy of $\G$ in either $C(X)\rtimes_r \G$ or $C(X)\rtimes \G$ by $(\delta_g)_{g\in \G}$. 

The next result is an easy consequence of \cite[Corollary 2.10]{MR568977}.

\begin{prop}\label{indkt}

Let $\Lambda$ be a finite index subgroup of a group $\G$. Then $$C(\Gamma/\Lambda)\rtimes_r \Gamma\simeq M_{\Gamma/\Lambda}(\C)\otimes C^*_r(\Lambda).$$
\end{prop}

\begin{proof}
Take representatives $g_1,\dots,g_n$ for $\Gamma/\Lambda$. For $1\leq i \leq n$ and $g\in\G$, let $h(i,g)\in\Lambda$ and $\sigma_g(i)\in\{1,\dots,n\}$ be such that 
\begin{equation*}
gg_i=g_{\sigma_g(i)}h(i,g).
\end{equation*}
 Then there is a surjective $*$-homomorphism $\psi\colon C(\Gamma/\Lambda)\rtimes \Gamma\to M_{\Gamma/\Lambda}(\C)\otimes C^*_r(\Lambda)$ such that 
\begin{align*}
\psi(1_{\{l\}})=e_{l,l}\otimes 1 \text{ and } \psi(\delta_g)=\sum e_{\sigma_g(i),i}\otimes \delta_{h(i,g)},\text{ for $l\in\G/\Lambda$ and $g\in\G$.} 
\end{align*}

Let $\tau$ be the canonical faithful tracial state on $C^*_r(\Lambda)$ and $\varphi\colon M_{\Gamma/\Lambda}(\C)\to C(\G/\Lambda)$ be the canonical conditional expectation. Then $\varphi\otimes\tau\colon M_{\Gamma/\Lambda}(\C)\otimes C^*_r(\Lambda)\to C(\G/\Lambda)$ is a faithful conditional expectation such that $(\varphi\otimes\tau)\circ\psi$ is the canonical conditional expectation from $C(\Gamma/\Lambda)\rtimes \Gamma$ onto $C(\G/\Lambda)$. In particular, $\psi$ factors through an isomorphism from $C(\Gamma/\Lambda)\rtimes_r \Gamma$ into $M_{\Gamma/\Lambda}(\C)\otimes C^*_r(\Lambda)$.

\end{proof}

Hence, given an odometer $\G\curvearrowright X=\varprojlim\G/\G_i$, we have that $C(X)\rtimes_r \Gamma\simeq \varinjlim C(\G/\G_i)\rtimes_r\G\simeq\varinjlim M_{\Gamma/\Gamma_i}(\C)\otimes C^*_r(\Gamma_i)$ (this was observed already in \cite{MR2772352}). Therefore,
\begin{equation}
 K_*(C(X)\rtimes_r \Gamma)\simeq\varinjlim K_*(C^*_r(\Gamma_i)).\label{kt}
\end{equation}

\subsection{Homology of odometers} 

Given a group $\G$, we denote by $\G'$ its \emph{commutator subgroup}. Recall that $\G_{\mathrm{ab}}:=\frac{\G}{\G'}$ is the \emph{abelianization} of $\G$.

Let $\Lambda$ be a finite index subgroup of a group $\G$. Let us recall the definition of the \emph{transfer map} $\mathrm{tr}_\Lambda^\G\colon H_*(\G)\to H_*(\Lambda)$ (see, e.g., \cite[Section III.9]{MR1324339}).

Consider $\Z$ as a trivial $\Lambda$-module and $\Z^{\G/\Lambda}$ as a permutation $\G$-module. Let $i\colon \Lambda\to \G$ be the inclusion and $\varphi\colon\Z\to\Z^{\G/\Lambda}$ be the additive map given by $\varphi(1):=\delta_\Lambda$. Then functoriality of $H_*$ gives a homomorphism $(i,\varphi)_*\colon H_*(\Lambda)\to H_*(\G,\Z^{\G/\Lambda})$, which, by Shapiro's lemma, is an isomorphism. 

Also let $\psi\colon\Z\to\Z^{\G/\Lambda}$ be the additive map given by 
\begin{equation*}
\psi(1):=\sum_{x\in\G/\Lambda}\delta_x.
\end{equation*} 
Note that $\psi$ is $\G$-equivariant, hence it induces a homomorphism $H_*(\psi)\colon H_*(\G)\to H_*(\G,\Z^{\G/\Lambda})$. Define $\mathrm{tr}_\Lambda^\G:=(i,\varphi)_*^{-1}\circ H_*(\psi)\colon H_*(\G)\to H_*(\Lambda)$.

For $*=0$, $\mathrm{tr}_\Lambda^\G\colon \Z\to\Z$ is multiplication by $[\G:\Lambda]$.

Take representatives $g_1,\dots,g_n$ for $\G/\Lambda$. For $1\leq i \leq n$ and $g\in\G$, let $h(i,g)\in\Lambda$ be such that $gg_i=g_{\sigma_g(i)}h(i,g)$ for some $\sigma_g(i)\in\{1,\dots,n\}$. Then, for $*=1$, $\mathrm{tr}_\Lambda^\G\colon \G_\mathrm{ab}\to \Lambda_\mathrm{ab}$ is given by
\begin{equation}
\mathrm{tr}_\Lambda^\G(g\G' )=\sum_{i=1}^n h(i,g)\Lambda'.\label{homo1}
\end{equation}

\begin{proposition}\label{transfer}
Let $\G\curvearrowright X=\varprojlim\G/\G_i$ be an odometer. Then

\begin{align*}
H_*(\Gamma,C(X,\mathbb{Z}))\simeq\varinjlim (H_*(\Gamma_i),\mathrm{tr}_{\G_{i+1}}^{\G_i}).
\end{align*}

\end{proposition}
\begin{proof}
For $i\geq 1$, let $p_i\colon\G/\G_{i+1}\to\G/\G_i$ be as in \eqref{proje} and $q_i\colon\Z^{\G/\G_i}\to\Z^{\G/\G_{i+1}}$ be given by $q_i(f):=f\circ p_i$, for $f\colon\G/\G_i\to\Z$. Note that $C(X,\Z)$ and $\varinjlim (\Z^{\G/\G_i},q_i)$ are isomorphic as $\G$-modules. Since homology commutes with direct limits, we obtain that $$H_*(\Gamma,C(X,\mathbb{Z}))\simeq\varinjlim (H_*(\Gamma,\mathbb{Z}^{\Gamma/\Gamma_i}),H_*(q_i)).$$

Furthermore, for every $i\geq 1$, Shapiro's lemma gives an isomorphism $$\varphi_i\colon H_*(\G_i)\to H_*(\G,\Z^{\G/\G_i}).$$ We claim that $ H_*(q_i)\circ\varphi_i=\varphi_{i+1}\circ\mathrm{tr}_{\G_{i+1}}^{\G_i}$.

Indeed, let $j\colon\G_i\to\G$ be the inclusion, $\psi\colon\Z\to\Z^{\G_i/\G_{i+1}}$ be the additive map given by $\psi(1):=\sum_{x\in\G_i/\G_{i+1}}\delta_x$ and $\theta\colon\Z^{\G_i/\G_{i+1}}\to\Z^{\G/\G_{i+1}}$ be the additive map given by $\theta(\delta_x)=\delta_x$, for $x\in\G_i/\G_{i+1}$.

One can readily check that $\varphi_{i+1}\circ\mathrm{tr}_{\G_{i+1}}^{\G_i}=(j,\theta)_*\circ H_*(\psi)=H_*(q_i)\circ\varphi_i$. This concludes the proof of the proposition.
\end{proof}

In particular, for an odometer $\G\curvearrowright X=\varprojlim\G/\G_i$, we have that
\begin{equation}
H_0(\G,C(X,\Z))\simeq\left\{\frac{m}{[\G:\G_i]}\in\Q:m\in\Z,i\geq 1\right\}.\label{homo0}
\end{equation}

\subsection{Ample groupoids}

A topological groupoid $G$ is said to be \emph{ample} if $G$ is locally compact, Hausdorff, \'etale 
(in the sense that the range and source maps $r,s\colon G\to G$ are local homeomorphisms onto the unit space $G^{(0)}$) 
and the unit space $G^{(0)}$ is totally disconnected.

Let $G$ be an ample groupoid. The \emph{orbit} of an $x\in G^{(0)}$ is the set $r(s^{-1}(x))$, and the groupoid is said to be \emph{minimal} if the orbit of each point of $G^{(0)}$ is dense in $G^{(0)}$.

A \emph{bisection} is a subset $S\subset G$ such that $r|_S$ and $s|_S$ are injective. Note that, if $S$ is open, then $r|_S$ and $s|_S$ are homeomorphisms onto their images. 

Now assume that $G$ is an ample groupoid with compact unit space. The \emph{topological full group} $[[G]]$ is the group of compact-open bisections $U$ such that $r(U)=s(U)=G^{(0)}$.

There is a homomorphism $\theta$ from $[[G]]$ to the group of homeomorphisms of $G^{(0)}$, 
given by $\theta_U:=r\circ (s|_U)^{-1}$. If $G$ is \emph{essentially principal} (that is, $\mathrm{int}\{g\in G:r(g)=s(g)\}=G^{(0)}$), then  $\theta$ is injective.

\subsection{Homology of ample groupoids}In this sub-section, we recall the definition of the homology groups of an ample groupoid (\cite[Definition 2.3]{MR3552533}, see also the beginning of \cite[Section 4]{farsi2018ample}). 

Given locally compact, Hausdorff spaces $X$ and $Y$, a local homeomorphism $\pi\colon X\to Y$ induces a homomorphism $\pi_*\colon C_c(X,\Z)\to C_c(Y,\Z)$ given by $$\pi_*(f)(y):=\sum_{x\in\pi^{-1}(y)} f(x).$$

Given an ample groupoid $G$, and $n\geq 1$, let $G^{(n)}$ be the space of sequences $(g_1,\dots,g_n)\in G^n$ such that the product $g_1\dots g_n$ is well-defined. The topology in $G^{(n)}$ is the one inherited from the product topology in $G^n$.

For $i=0, \dots,n$, let $d_i\colon G ^{(n)}\to G^{(n-1)}$ be given by

\begin{align*}
d_i(g_1, \dots,g_n):=
\begin{cases}
(g_2,\dots,g_n) & \text{if $i=0$}\\
(g_1,\dots,g_ig_{i+1},\dots,g_n) & \text{if $1\leq i\leq n-1$}\\
(g_1,\dots,g_{n-1}) & \text{if $i=n$}. 
\end{cases}
\end{align*} 

If $n=1$, let $d_0,d_1\colon G^{(1)}\to G^{(0)}$ be the source and range maps, respectively. 

Clearly, the maps $d_i$ are local homeomorphisms. Define $\delta_n\colon C_c(G^{(n)},\Z)\to C_c(G^{(n-1)})$ by $$\delta_n:=\sum_{i=0}^n (-1)^id_{i_*}.$$ Then

$$0\xleftarrow{\delta_0}C_c(G^{(0)},\Z)\xleftarrow{\delta_1} C_c(G^{(1)},\Z)\xleftarrow{\delta_2}
\dots$$
is a chain complex. Denote by $H_n(G):=\frac{\ker\delta_n}{\Ima\delta_{n+1}}$ its homology groups.

\begin{example}\label{tg}
Let $\varphi$ be an action of a group $\Gamma$ on a compact, totally disconnected, Hausdorff space $X$. 
As a space, the \emph{transformation groupoid} $G$ associated with $\varphi$ is $G:=\Gamma\times X$ equipped with the product topology.
The product of two elements $(h,y), (g,x)\in G$ is defined if and only if $y=gx$, in which case $(h,gx)(g,x):=(hg,x)$.
Inversion is given by $(g,x)^{-1}:=(g^{-1},gx)$. The unit space $G^{(0)}$ is naturally identified with $X$ and $G$ is ample. Also, $G$ is essentially principal if and only if $\varphi$ is topologicall free, and it holds that $C^*_r(G)\simeq C(X)\rtimes_r\G$ and $H_*(G)\simeq H_*(\G,C(X,\Z))$.

Given $g\in \G$ and $A\subset X$, notice that $s(\{g\}\times A)=A$, and $r(\{g\}\times A)=gA$.

Let us now describe the topological full group $[[G]]$. Take $g_1,\dots,g_n\in\Gamma$ and $A_1,\dots,A_n\subset X$ clopen sets such that $X = \sqcup_{i=1}^n A_i = \sqcup_{i=1}^n g_iA_i$ (disjoint unions). Then $U:=\cup_{i=1}^n\{g_i\}\times A_i$ is a compact-open bisection and $s(U)=r(U)=X$. Therefore, $U\in[[G]]$. Conversely, it is easy to see that any $U\in[[G]]$ is as above.
\end{example}

\section{Counterexamples to the HK conjecture}\label{count}

In \cite{MR3552533}, Matui conjectured (\emph{HK conjecture}) that, given a second countable, étale, minimal, essentially principal groupoid $G$ with unit space homeomorphic to the Cantor set, the following holds:

\begin{equation*}
K_*(C^*_r(G))\simeq\bigoplus_{k\geq 0} H_{2k+*}(G) \text{ for } *=0,1.
\end{equation*}

As in Example \ref{ce}, let $\G\curvearrowright X=\varprojlim\G/\G_i$, with $\G:=\Z\rtimes\Z_2$ and, for $i\geq 1$, $\G_i:=n_i\Z\rtimes\Z_2$, where $(n_i)$ is a strictly increasing sequence of natural numbers such that, for $i\geq 1$, $n_i|n_{i+1}$. We are going to compute $K_*(C(X)\rtimes\G)$ and $H_*(\G,C(X,\Z))$, and conclude that the transformation groupoids associated to these odometers are counterexamples to the HK conjecture.

Note that $\G/\G_i$ can be identified with the abelian group $\Z_{n_i}=\{0,\dots,n_i-1\}$. The action of $(1,0)\in\G$ on an element $(x_i)\in\varprojlim\Z_{n_i}$ is given by summing $1$ in each entry, and the action of $(0,1)\in\G$ is given by multiplying each entry by $-1$.

Since $\G\simeq\G_i$ for every $i$ and $K_1(C^*(\Z\rtimes\Z_2))= 0$ (\cite[10.11.5(a)]{MR859867}), it follows from \eqref{kt} that $K_1(C(X)\rtimes \G)=0$. In fact, $C(X)\rtimes \G$ is an AF algebra (see \cite{MR1245825} or \cite{MR962104}).

One can also compute $K_0(C(X)\rtimes\G)$ by applying \eqref{kt}, but we will instead make use of the following result of Bratteli, Evans and Kishimoto (\cite[Theorem 4.1]{MR1245825}):

\begin{theorem}\label{fix}
Given an action of $\Z\rtimes\Z_2$ on the Cantor set $X$ such that the restricted $\Z$-action is minimal and $(0,1),(1,1)\in\Z\rtimes\Z_2$ have at most a finite number $m_{(0,1)}$ and $m_{(1,1)}$ of fixed points, with $m_{(0,1)}+m_{(1,1)}>0$, then $K_0(C(X)\rtimes\Z\rtimes\Z_2)$ is isomorphic to $$(1+(0,1)_*)\left(\frac{C(X,\Z)}{(1-(1,0)_*)(C(X,\Z))}\right)\oplus\Z^{m_{(0,1)} + m_{(1,1)}}.$$

\end{theorem}

Actually, the hypothesis that $(0,1)$ or $(1,1)$ must have at least one fixed point does not appear in \cite[Theorem 4.1]{MR1245825}, but Thomsen showed in \cite{MR2586354} that the theorem is false without this hypothesis.

In order to apply Theorem \ref{fix} to the odometers of Example \ref{ce}, we need to compute the number of fixed points of $(0,1),(1,1)\in\Z\rtimes\Z_2$:

\begin{lemma}\label{fixed}
Let $\G\curvearrowright \varprojlim\G/\G_i$ as in Example \ref{ce}, and denote by $m_{(0,1)}$ and $m_{(1,1)}$ the number of fixed points of $(0,1)$ and $(1,1)$. Then
\begin{align*}
m_{(0,1)}&=\begin{cases}1 &\text{if $\frac{n_{i+1}}{n_i}$ is even for infinitely many $i$}\\
1 & \text{if $n_i$ is odd for every $i$}\\
2 & \text{otherwise,}\end{cases}\\
m_{(1,1)}&=\begin{cases}
1 & \text{if $n_i$ is odd for every $i$}\\
0 & \text{otherwise.}\end{cases}
\end{align*}

\end{lemma}
\begin{proof}

An element $(x_i)\in\varprojlim\Z_{n_i}$ is a fixed point of $(0,1)$ if and only if, for every $i$, $x_i=0$ or $x_i=\frac{n_i}{2}$. Likewise, $(x_i)$ is a fixed point of $(1,1)$ if and only if $x_i=\frac{n_i+1}{2}$ for every $i$.

Therefore, if $n_i$ is odd for every $i$, then the only fixed point of $(0,1)$ is $(0)$ and the only fixed point of $(1,1)$ is $(\frac{n_i+1}{2})$. 

If $\frac{n_{i+1}}{n_i}$ is even for infinitely many $i$, then $(1,1)$ admits no fixed points and the only fixed point of $(0,1)$ is $(0)$.

Finally, if there exists $i_0$ such that $n_{i_0}$ is even, but $\frac{n_{i+1}}{n_i}$ is odd for every $i\geq i_0$, and $i_0$ is minimal with these properties, then the only fixed points of $(0,1)$ are $(0)$ and $(0,\dots,0,\frac{n_{i_0}}{2},\frac{n_{i_0+1}}{2}\dots)$. Furthermore, $(1,1)$ does not admit fixed points.

\end{proof}

\begin{proposition}\label{kate}

Let $\G\curvearrowright X=\varprojlim\G/\G_i$ as in Example \ref{ce}. Then
\begin{align*}
K_0(C(X)\rtimes\G)\simeq\begin{cases}\{\frac{m}{n_i}:m\in\Z,i\geq 1\}\oplus\Z &\text{if $\frac{n_{i+1}}{n_i}$ is even for infinitely many $i$}\\
\{\frac{m}{n_i}:m\in\Z,i\geq 1\}\oplus\Z^2 & \text{otherwise.}\end{cases}
\end{align*}

\end{proposition}
\begin{proof}

In order to apply Theorem \ref{fix}, let us first compute $\frac{C(X,\Z)}{(1-(1,0)_*)(C(X,\Z))}$.

Consider the odometer $\Z\curvearrowright X= \varprojlim\Z_{n_i}$. Notice that $\frac{C(X,\Z)}{(1-(1,0)_*)(C(X,\Z))}\simeq H_0(\Z,C(X,\Z))$. Therefore, by \eqref{homo0}, $$\frac{C(X,\Z)}{(1-(1,0)_*)(C(X,\Z))}\simeq\left\{\frac{m}{n_i}:m\in\Z,i\geq 1\right\}.$$

Given $j\geq 1$ and $k\in\Z_{n_j}$, let $U(j,k):=\{(x_i)\in X:x_j=k\}$. Observe that $((1,0)_*)^k(1_{U(j,0)})=1_{U(j,k)}$, and  $(0,1)_*(1_{U(j,k)})=1_{U(j,-k)}$. From these facts, it follows that $(0,1)_*$ acts trivially on $\frac{C(X,\Z)}{(1-(1,0)_*)(C(X,\Z))}$.

The result now is a consequence of Theorem \ref{fix} and Lemma \ref{fixed}.

\end{proof}

The following lemma will be useful for computing the homology of the $(\Z\rtimes\Z_2)$-odometers that we are investigating in this section.

\begin{lemma}\label{techi}
Let $\sigma$ be an involutive homeomorphism on a compact, totally disconnected, Hausdorff space $Y$ such that $F_\sigma:=\{y\in Y:\sigma(y)=y\}$ is finite. Then, for $k\geq 0$, $H_{2k+1}(\Z_2,C(Y,\Z))=(\Z_2)^{F_\sigma}$.
\end{lemma}
\begin{proof}

By \cite[Theorem 6.2.2]{MR1269324}, we have that $$H_{2k+1}(\Z_2,C(Y,\Z))=\frac{\{f\in C(Y,\Z):f\circ\sigma=f\}}{\{f+f\circ\sigma:f\in C(Y,\Z)\}}.$$ 

Let $E\colon\frac{\{f\in C(Y,\Z):f\circ\sigma=f\}}{\{f+f\circ\sigma:f\in C(Y,\Z)\}}\to(\Z_2)^{F_\sigma}$ be given by evaluation at the points of $F_\sigma$. Clearly, $E$ is a well-defined homomorphism, and we will show that it is bijective.

For each $y\in F_\sigma$, take $A_y\subset Y$ clopen set such that $A_y\cap F_\sigma=\{y\}$, and $\sigma(A_y)=A_y$. Note that $E([1_{A_y}])=\delta_y$, for every $y\in F_\sigma$. Hence, $E$ is surjective.

Let us now verify injectivity of $E$. Take $f\in C(Y,\Z)$ such that $f\circ\sigma=f$ and $E([f])=0$. We will show that $[f]=0$. Clearly, we can assume that $f|_{F_\sigma}=0$. Then $f$ can be written as a linear combination of functions of the form $1_B+1_{\sigma(B)}$, for certain $B\subset Y$ clopen sets such that $B\cap\sigma(B)=\emptyset$. This concludes the proof of the lemma.

\end{proof} 

\begin{theorem}
Let $\G\curvearrowright X=\varprojlim\G/\G_i$ as in Example \ref{ce}. Then, for $k\geq 1$, 
\begin{align*}
H_0(\G,C(X,\Z))&=\left\{\frac{m}{n_i}:m\in\Z,i\geq 1\right\}\\
H_{2k}(\G,C(X,\Z))&= 0,\\
H_{2k-1}(\G,C(X,\Z))&=\begin{cases}
\Z_2 & \text{if $\frac{n_{i+1}}{n_i}$ is even for infinitely many $i$}\\
(\Z_2)^2 & \text{otherwise.}\end{cases}
\end{align*}

\end{theorem}
\begin{proof}
As $\Z\rtimes\Z_2\simeq \Z_2 * \Z_2$, it follows from \cite[Corollary 6.2.10]{MR1269324} and Proposition \ref{transfer} that, for $k\geq 1$, $H_{2k}(\G,C(X,\Z))=0$. Moreover, $H_0(\G,C(X,\Z))=\{\frac{m}{n_i}:m\in\Z,i\geq 1\}$ (see \eqref{homo0}). 

Furthermore, from Lemmas \ref{fixed} and \ref{techi}, it follows that, for $k\geq 1$, $$H_{2k+1}(\G,C(X,\Z)) = (\Z_2)^r,$$ where $r=1$ if $\frac{n_{i+1}}{n_i}$ is even for infinitely many $i$, and $r=2$ otherwise.

Finally, we compute $H_1(\G,C(X,\Z))$.
For every $i$, $(\G_i)_\mathrm{ab}\simeq\Z_2\times\Z_2$ with canonical generators $(n_i,0)\G_i'$ and $(0,1)\G_i'$. Under these identifications, and using \eqref{homo1}, it is easy to see that $\mathrm{tr}_{\G_{i+1}}^{\G_i}\colon\Z_2\times\Z_2\to\Z_2\times\Z_2$ is given by $\mathrm{tr}_{\G_{i+1}}^{\G_i}(1,0)=(1,0)$ and 

\begin{align*}
\mathrm{tr}_{\G_{i+1}}^{\G_i}(0,1)=\begin{cases}(1,0) &\text{if $\frac{n_{i+1}}{n_i}$ is even}\\
(0,1) & \text{otherwise.}\end{cases}
\end{align*}

 It follows from Proposition \ref{transfer} that $H_1(\G,C(X,\Z))=(\Z_2)^r$, where $r=1$ if $\frac{n_{i+1}}{n_i}$ is even for infinitely many $i$, and $r=2$ otherwise. 

\end{proof}

\section{On the AH conjecture for odometers}\label{ah}

Let $G$ be a second countable, ample, essentially principal and minimal groupoid with unit space homeomorphic to the Cantor set. Matui conjectured in \cite{MR3552533} (\emph{AH conjecture}) that there is an exact sequence

\begin{align*}
H_0(G)\otimes\Z_2\xrightarrow{j}[[G]]_{\mathrm{ab}}\xrightarrow{I}H_1(G)\rightarrow 0.
\end{align*}
The map $I\colon [[G]]_\mathrm{ab}\to H_1(G)$ is given by $I(U[[G]]'):=[1_U]$, for $U\in[[G]]$ (this is defined for any ample groupoid with compact unit space).

 Let us describe $j\colon H_0(G)\otimes\Z_2\to[[G]]_{\mathrm{ab}}$, for $G$ an ample groupoid with compact unit space such that the orbit of each point of $G^{(0)}$ has at least three points. Given $F\subset G$ compact-open bisection such that $s(F)\cap r(F)=\emptyset$, notice that $\tau_F:=F\cup F^{-1}\cup(G^{(0)}\setminus(s(F)\cup r(F)))\in[[G]]$. The map $j$ is the unique homomorphism such that $j([1_{s(F)}]\otimes 1)=\tau_F[[G]]'$ for each such $F$. See \cite[Theorem 7.2]{nekrashevych2017simple} for a proof that $j$ is well-defined. It is easy to see that $I\circ j = 0$.

Given a set $X$, we denote by $S_X$ the group of bijections on $X$.

The proof of the following lemma uses a technique employed in \cite[Proposition 2.1]{MR3103094} and \cite[Proposition 4.6]{MR3556453}.

\begin{lemma}\label{fini}

Let $\Lambda$ be a finite index subgroup of $\G$ and $G$ the transformation groupoid associated with $\G\curvearrowright\G/\Lambda$. Then $[[G]]\simeq\Lambda^{\G/\Lambda}\rtimes S_{\G/\Lambda}$ and $[[G]]_{\mathrm{ab}}\simeq \Lambda_{\mathrm{ab}}\times (S_{\G/\Lambda})_{\mathrm{ab}}$.

\end{lemma}

\begin{proof}
There is an epimorphism $\theta\colon[[G]]\to S_{\G/\Lambda}$ given by $\theta_U:=r\circ (s|_U)^{-1}$, for $U\in [[G]]$. Fix representatives $g_1,\dots,g_n$ for $\G/\Lambda$ and define $\eta\colon S_{\G/\Lambda}\to [[G]]$ by $\eta(\sigma):=\cup_{i=1}^n \{(g_{\sigma(i)}g_i^{-1},g_i\Lambda)\}$, for $\sigma\in S_{\G/\Lambda}$. Note that $\eta$ is a homomorphism which is a right inverse for $\theta$. In particular, $[[G]]\simeq\ker\theta\rtimes S_{\G/\Lambda}$. Clearly, $$\ker\theta\simeq g_1\Lambda g_1^{-1}\times\dots\times g_n\Lambda g_n^{-1}\simeq \Lambda^{\G/\Lambda}.$$

Therefore, there is an isomorphism from $\Lambda_\mathrm{ab}\times (S_{\G/\Lambda})_\mathrm{ab}$ into $[[G]]_\mathrm{ab}$ induced by the embeddings $\eta\colon S_{\G/\Lambda}\to [[G]]$, defined in the previous paragraph, and $\zeta\colon\Lambda\to[[G]]$ given by
$\zeta(\lambda):=\{(\lambda,\Lambda)\}\cup(\{e\}\times\{\Lambda\}^c)$, for $\lambda\in\Lambda$.
\end{proof}

\begin{lemma}\label{finito}

Let $\Lambda$ be a finite index subgroup of $\G$ such that $[\G:\Lambda]\geq 3$, and $G$ the transformation groupoid associated with $\G\curvearrowright\G/\Lambda$. Then the following sequence is split exact:

\begin{align*}
0\rightarrow H_0(G)\otimes\Z_2\xrightarrow{j}[[G]]_{\mathrm{ab}}\xrightarrow{I}H_1(G)\rightarrow 0.
\end{align*}

\end{lemma}

\begin{proof}

Identify $[[G]]_{\mathrm{ab}}$ with $\Lambda_{\mathrm{ab}}\times (S_{\G/\Lambda})_{\mathrm{ab}}$ as in the previous lemma. Note that $(S_{\G/\Lambda})_\mathrm{ab}\simeq \Z_2$ and, by Shapiro's lemma, $H_1(G)\simeq \Lambda_\mathrm{ab}$ and $H_0(G)\otimes\Z_2\simeq\Z_2$. Under these identifications, $I\colon\Lambda_\mathrm{ab}\times \Z_2\to\Lambda_\mathrm{ab}$ is the canonical projection, and $j\colon\Z_2\to\Lambda_\mathrm{ab}\times \Z_2$ is the canonical inclusion.

\end{proof}

\begin{theorem}\label{homodo}

Let $G$ be the transformation groupoid associated to an odometer $\G\curvearrowright\varprojlim\G/\G_i$. The following sequence is exact.

\begin{align}
0\rightarrow H_0(G)\otimes\Z_2\xrightarrow{j}[[G]]_{\mathrm{ab}}\xrightarrow{I}H_1(G)\rightarrow 0.\label{es}
\end{align}
\end{theorem}
\begin{proof}

Given $i\in\N$, let $G_i$ be the transformation groupoid associated to $\G\curvearrowright\G/\G_i$. The result follows from the fact that $H_*(G)=\varinjlim H_*(G_i)$, $[[G]]_\mathrm{ab}=\varinjlim[[G_i]]_\mathrm{ab}$ and Lemma \ref{finito}.

\end{proof}

Note that, given an odometer $\G\curvearrowright X=\varprojlim\G/\G_i$, it holds that 
\begin{equation*}
H_0(\G,C(X,\Z))\otimes\Z_2\simeq\begin{cases}
0 & \text{if $[\G_{i}:\G_{i+1}]$ is even for infinitely many $i$}\\
\Z_2 & \text{otherwise.}\end{cases}
\end{equation*}

We do not know whether there is an odometer for which the exact sequence \eqref{es} does not split.

\bibliographystyle{acm}
\bibliography{bibliografia}

\begin{thebibliography}{10}

\bibitem{MR859867}
{\sc Blackadar, B.}
\newblock {\em {$K$}-theory for operator algebras}, vol.~5 of {\em Mathematical
  Sciences Research Institute Publications}.
\newblock Springer-Verlag, New York, 1986.

\bibitem{MR1245825}
{\sc Bratteli, O., Evans, D.~E., and Kishimoto, A.}
\newblock Crossed products of totally disconnected spaces by {$Z_2\ast Z_2$}.
\newblock {\em Ergodic Theory Dynam. Systems 13}, 3 (1993), 445--484.

\bibitem{MR1324339}
{\sc Brown, K.~S.}
\newblock {\em Cohomology of groups}, vol.~87 of {\em Graduate Texts in
  Mathematics}.
\newblock Springer-Verlag, New York, 1994.
\newblock Corrected reprint of the 1982 original.

\bibitem{MR3672904}
{\sc Brownlowe, N., Mundey, A., Pask, D., Spielberg, J., and Thomas, A.}
\newblock {$C^*$}-algebras associated to graphs of groups.
\newblock {\em Adv. Math. 316\/} (2017), 114--186.

\bibitem{MR2772352}
{\sc Carri\'{o}n, J.~R.}
\newblock Classification of a class of crossed product {$C^\ast$}-algebras
  associated with residually finite groups.
\newblock {\em J. Funct. Anal. 260}, 9 (2011), 2815--2825.

\bibitem{MR3556453}
{\sc Cortez, M. a.~I., and Medynets, K.}
\newblock Orbit equivalence rigidity of equicontinuous systems.
\newblock {\em J. Lond. Math. Soc. (2) 94}, 2 (2016), 545--556.

\bibitem{MR2427048}
{\sc Cortez, M.~I., and Petite, S.}
\newblock {$G$}-odometers and their almost one-to-one extensions.
\newblock {\em J. Lond. Math. Soc. (2) 78}, 1 (2008), 1--20.

\bibitem{MR3289281}
{\sc de~Cornulier, Y.}
\newblock Groupes pleins-topologiques (d'apr\`es {M}atui, {J}uschenko, {M}onod,
  {$\ldots$}).
\newblock {\em Ast\'erisque}, 361 (2014), Exp. No. 1064, viii, 183--223.

\bibitem{MR3506971}
{\sc Dyer, J., Hurder, S., and Lukina, O.}
\newblock The discriminant invariant of {C}antor group actions.
\newblock {\em Topology Appl. 208\/} (2016), 64--92.

\bibitem{farsi2018ample}
{\sc Farsi, C., Kumjian, A., Pask, D., and Sims, A.}
\newblock Ample groupoids: equivalence, homology, and {M}atui's {HK}
  conjecture.
\newblock {\em arXiv preprint arXiv:1808.07807\/} (2018).

\bibitem{MR568977}
{\sc Green, P.}
\newblock The structure of imprimitivity algebras.
\newblock {\em J. Funct. Anal. 36}, 1 (1980), 88--104.

\bibitem{MR2783933}
{\sc Ioana, A.}
\newblock Cocycle superrigidity for profinite actions of property ({T}) groups.
\newblock {\em Duke Math. J. 157}, 2 (2011), 337--367.

\bibitem{MR962104}
{\sc Kumjian, A.}
\newblock An involutive automorphism of the {B}unce-{D}eddens algebra.
\newblock {\em C. R. Math. Rep. Acad. Sci. Canada 10}, 5 (1988), 217--218.

\bibitem{MR3103094}
{\sc Matui, H.}
\newblock Some remarks on topological full groups of {C}antor minimal systems
  {II}.
\newblock {\em Ergodic Theory Dynam. Systems 33}, 5 (2013), 1542--1549.

\bibitem{MR3552533}
{\sc Matui, H.}
\newblock Étale groupoids arising from products of shifts of finite type.
\newblock {\em Adv. Math. 303\/} (2016), 502--548.

\bibitem{nekrashevych2017simple}
{\sc Nekrashevych, V.}
\newblock Simple groups of dynamical origin.
\newblock {\em Ergodic Theory and Dynamical Systems\/} (2017), 1--26.

\bibitem{MR3779956}
{\sc Nekrashevych, V.}
\newblock Palindromic subshifts and simple periodic groups of intermediate
  growth.
\newblock {\em Ann. of Math. (2) 187}, 3 (2018), 667--719.

\bibitem{MR2550195}
{\sc Orfanos, S.}
\newblock Generalized {B}unce-{D}eddens algebras.
\newblock {\em Proc. Amer. Math. Soc. 138}, 1 (2010), 299--308.

\bibitem{ortega2018homology}
{\sc Ortega, E.}
\newblock Homology of the {K}atsura-{E}xel-{P}ardo groupoid.
\newblock {\em arXiv preprint arXiv:1806.09297\/} (2018).

\bibitem{scarparo2017c}
{\sc Scarparo, E.}
\newblock On the {$C^*$}-algebra generated by the {K}oopman representation of a
  topological full group.
\newblock {\em arXiv preprint arXiv:1705.07665\/} (2017).

\bibitem{MR2586354}
{\sc Thomsen, K.}
\newblock The homoclinic and heteroclinic {$C^*$}-algebras of a generalized
  one-dimensional solenoid.
\newblock {\em Ergodic Theory Dynam. Systems 30}, 1 (2010), 263--308.

\bibitem{MR1269324}
{\sc Weibel, C.~A.}
\newblock {\em An introduction to homological algebra}, vol.~38 of {\em
  Cambridge Studies in Advanced Mathematics}.
\newblock Cambridge University Press, Cambridge, 1994.

\end{thebibliography}
\end{document}